\pdfoutput=1
\RequirePackage{ifpdf}
\ifpdf 
\documentclass[pdftex]{sigma}
\else
\documentclass{sigma}
\fi

\usepackage{mathrsfs,bm}

\numberwithin{equation}{section}

\newtheorem{Theorem}{Theorem}[section]
\newtheorem{Corollary}[Theorem]{Corollary}
\newtheorem{Lemma}[Theorem]{Lemma}
\newtheorem{Proposition}[Theorem]{Proposition}
 { \theoremstyle{definition}

\newtheorem{Remark}[Theorem]{Remark} }

\begin{document}

\newcommand{\arXivNumber}{1710.06534}

\renewcommand{\PaperNumber}{046}

\FirstPageHeading

\ShortArticleName{Lower Bounds for Numbers of Real Self-Dual Spaces in Problems of Schubert Calculus}

\ArticleName{Lower Bounds for Numbers of Real Self-Dual Spaces\\ in Problems of Schubert Calculus}

\Author{Kang LU}

\AuthorNameForHeading{K.~Lu}

\Address{Department of Mathematical Sciences, Indiana University~-- Purdue University Indianapolis,\\
402 North Blackford St, Indianapolis, IN 46202-3216, USA}
\Email{\href{mailto:lukang@iupui.edu}{lukang@iupui.edu}}

\ArticleDates{Received November 27, 2017, in final form May 07, 2018; Published online May 14, 2018}

\Abstract{The self-dual spaces of polynomials are related to Bethe vectors in the Gaudin model associated to the Lie algebras of types B and C. In this paper, we give lower bounds for the numbers of real self-dual spaces in intersections of Schubert varieties related to osculating flags in the Grassmannian. The higher Gaudin Hamiltonians are self-adjoint with respect to a nondegenerate indefinite Hermitian form. Our bound comes from the computation of the signature of this form.}

\Keywords{real Schubert calculus; self-dual spaces; Bethe ansatz; Gaudin model}

\Classification{14N99; 17B80; 82B23}

\section{Introduction}
It is well known that the problem of finding the number of real solutions to algebraic systems is very difficult, and not many results are known. In particular, the counting of real points in problems of Schubert calculus in
the Grassmannian has received a lot of attention, see \cite{EG,HHS,HS,HSZ,MT,SS,S} for example. In this paper, we give lower bounds for the numbers of real self-dual spaces in intersections of Schubert varieties related to osculating flags in the Grassmannian.

We define the Grassmannian $\operatorname{Gr}(N,d)$ to be the set of all $N$-dimensional subspaces of the $d$-dimensional space $\mathbb C_d[x]$ of polynomials in $x$ of degree less than $d$. In other words, we always assume for $X\in\operatorname{Gr}(N,d)$, we have $X\subset \mathbb C_d[x]$. Set $\mathbb P^1=\mathbb C\cup\{\infty\}$. Then, for any $z\in {\mathbb {P}}^1$, we have the osculating flag $\mathcal F(z)$, see~\eqref{F(inf)},~\eqref{F(z)}. Denote the Schubert cells corresponding to~$\mathcal F(z)$ by~$\Omega_{\xi}(\mathcal F(z))$, where $\xi=(d-N\geqslant \xi_1\geqslant \xi_2\geqslant \cdots \geqslant \xi_N\geqslant 0)$ are partitions. Then the set~$\Omega_{\bm\xi,\bm z}$ consists of spaces $X\in\operatorname{Gr}(N,d)$ such that $X$ belongs to the intersection of Schubert cells~$\Omega_{\xi^{(i)}}(\mathcal F(z_i))$ for $\bm z=(z_1,\dots,z_n)$ and $\bm\xi=\big(\xi^{(1)},\dots,\xi^{(n)}\big)$, where all $z_i\in{\mathbb {P}}^1$ are distinct and~$\xi^{(i)}$ are partitions, see~\eqref{Omega}. A~point $X\in\operatorname{Gr}(N,d)$ is called \emph{real} if it has a basis consisting of polynomials with all coefficients real. A lower bound for the number of real points in $\Omega_{\bm\xi,\bm z}$ is given in~\cite{MT}.

Let $X\in\operatorname{Gr}(N,d)$ be an $N$-dimensional subspace of polynomials in $x$. Let $X^\vee$ be the $N$-dimensional space of polynomials which are Wronskian determinants of $N-1$ elements of~$X$
\begin{gather*}
X^\vee=\big\{\det\big(d^{i-1}\varphi_j/dx^{i-1}\big)_{i,j=1}^{N-1}, \,\varphi_j(x)\in X \big\}.
\end{gather*}
The space $X$ is called \emph{self-dual} if $X^\vee= \psi\cdot X$ for some polynomial $\psi(x)$, see \cite{MV}. We define~$\mathrm s\Omega_{\bm\xi,\bm z}$ the subset of $\Omega_{\bm\xi,\bm z}$ consisting of all self-dual spaces. Our main result of this paper is a~lower bound for the number of real self-dual spaces in $\Omega_{\bm\xi,\bm z}$, see Corollary~\ref{bound thm}, i.e., a lower bound for the number of real points in $\mathrm s\Omega_{\bm\xi,\bm z}$, by following the idea of~\cite{MT}.

Let $\mathfrak g_N$ be the Lie algebra $\mathfrak{so}_{2r+1}$ if $N=2r$ or the Lie algebra $\mathfrak{sp}_{2r}$ if $N=2r+1$. We also set $\mathfrak g_3=\mathfrak{sl}_2$. It is known from \cite{LMV}, see also \cite[Section~6.1]{MV}, that if $\mathrm s\Omega_{\bm\xi,\bm z}$ is nonempty, then $\xi_i^{(s)}-\xi_{i+1}^{(s)}=\xi_{N-i}^{(s)}-\xi_{N-i+1}^{(s)}$ for $i=1,\dots,N-1$. Hence the $\mathfrak{sl}_N$-weight corresponding to the partition $\xi^{(s)}$ has certain symmetry and thus induces a $\mathfrak g_N$-weight $\lambda^{(s)}$, cf.~\eqref{A to BC}. Therefore, the sequence of partitions $\bm\xi$ with nonempty $\mathrm s\Omega_{\bm\xi,\bm z}$ can be expressed in terms of a sequence of dominant integral $\mathfrak g_N$-weights $\bm\lambda=\big(\lambda^{1},\dots,\lambda^{(n)}\big)$ and a sequence of nonnegative integers $\bm k=(k_1,\dots,k_n)$, see Lemma~\ref{lem sym weight new}. In particular, $k_i=\xi_N^{(i)}$. We call $\bm\xi,\bm z$ or $\bm\lambda,\bm k,\bm z$ the \emph{ramification data}.

As a subset of $\Omega_{\bm\xi,\bm z}$, $\mathrm s\Omega_{\bm\xi,\bm z}$ can be empty even if $\Omega_{\bm\xi,\bm z}$ is infinite. However, if $\mathrm s\Omega_{\bm\xi,\bm z}$ is nonempty, then $\mathrm s\Omega_{\bm\xi,\bm z}$ is finite if and only if $\Omega_{\bm\xi,\bm z}$ is finite. More precisely, if
\begin{gather*}
|\bm\xi|:=\sum_{i=1}^n\big|\xi^{(i)}\big|=N(d-N),
\end{gather*}
then the number of points in $\mathrm s\Omega_{\bm\xi,\bm z}$ counted with multiplicities equals the multiplicity of the trivial $\mathfrak g_N$-module in the tensor product $V_{\lambda^{(1)}}\otimes\dots\otimes V_{\lambda^{(n)}}$ of irreducible $\mathfrak g_N$-modules of highest weights $\lambda^{(1)},\dots,\lambda^{(n)}$. Since we are interested in the counting problem, from now on, we always assume that $|\bm\xi|=N(d-N)$.

For brevity, we consider $\infty$ to be real. If all $z_1,\dots, z_n$ are real, it follows from \cite[Theorem~1.1]{MTV3} that all points in $\mathrm s\Omega_{\bm\xi,\bm z}$ are real. Hence the number of real points is maximal possible in this case. Moreover, it follows from \cite[Corollary 6.3]{MTV4} that all points in $\mathrm s\Omega_{\bm\xi,\bm z}$ are multiplicity-free.

Then we want to know how many real points we can guarantee in other cases. In general, a~necessary condition for the existence of real points is that the set $\{z_1,\dots,z_n\}$ should be invariant under the complex conjugation and the partitions at the complex conjugate points are the same. In other words, $\big(\lambda^{(i)},k_i\big)=\big(\lambda^{(j)},k_j\big)$ provided $z_i=\bar{z}_j$. In this case we say that~$\bm z$,~$\bm\lambda$,~$\bm k$ are \emph{invariant under conjugation}. Moreover, the greatest common divisor of $X\in\mathrm s\Omega_{\bm\xi,\bm z}$ in this case is a real polynomial. Hence we reduce the problem to the case that $k_i=0$, for all $i=1,\dots,n$.

The derivation of the lower bounds is based on the identification of the self-dual spaces of polynomials with points of spectrum of higher Gaudin Hamiltonians of types B and C ($\mathfrak g_N$, $N\geqslant 4$) built in \cite{LMV} and \cite{MV}, see Theorem \ref{bi rep sgr}. We show that higher Gaudin Hamiltonians of types B and C have certain symmetry with respect to the Shapovalov form which is positive definite Hermitian, see Proposition \ref{sym bethe}. In particular, these operators are self-adjoint with respect to the Shapovalov form for real $z_1,\dots,z_n$ and hence have real eigenvalues. Therefore, it follows from Theorem \ref{bi rep sgr} that self-dual spaces with real $z_1,\dots,z_n$ are real.

If some of $z_1,\dots,z_n$ are not real, but the data $\bm z$, $\bm\lambda$, $\bm k$ are invariant under the complex conjugation, the higher Gaudin Hamiltonians are self-adjoint with respect to a nondegenerate (indefinite) Hermitian form. One of the key observations for computing the lower bound for the number of real points in $\mathrm s\Omega_{\bm\xi,\bm z}$ is the fact that the number of real eigenvalues of such operators is at least the absolute value of the signature of the Hermitian form, see Lemma \ref{linalg2}.

The computation of the signature of the form is reduced to the computation of the character values of products of symmetric groups on products of commuting transpositions. The formula for such character, similar to the Frobenius formula in \cite{F} and \cite[Proposition 2.1]{MT}, is given in Proposition \ref{char prop}. Consequently, we obtain our main result, a lower bound for the number of real points in $\mathrm s\Omega_{\bm\xi,\bm z}$ for $N\geqslant 4$, see Corollary \ref{bound thm}. The case $N=2$ is the same as that of \cite{MT} since every 2-dimensional space of polynomials is self-dual. By the proof of \cite[Theorem 4.19]{LMV}, the case $N=3$ is reduced to the case of \cite{MT}, see Section \ref{sec N=2,3}.

Based on the identification of the self-self-dual spaces of polynomials with points of spectrum of higher Gaudin Hamiltonians of type G$_2$ built in \cite{BM} and \cite{LM}, we expect that lower bounds for the numbers of real self-self-dual spaces in $\Omega_{\bm\xi,\bm z}$ with $N=7$ can also be given in a similar way as conducted in this paper.

It is also interesting to find an algorithm to compute all (real) self-dual spaces with prescribed ramification data. The solutions to the Bethe ansatz equations described in \cite{LMV1} can be used to find nontrivial examples of self-dual spaces.

The paper is organized as follows. We start with the standard notation of Lie theory in Section~\ref{lie sec} and computations of characters of a product of symmetric groups in Section~\ref{char sec}. Then we recall notation and definitions for osculating Schubert calculus and self-dual spaces in Section~\ref{schubert sec}. In Section \ref{gaudin sec} we recall the connections between Gaudin model of types B, C and self-dual spaces of polynomials. The symmetry of higher Gaudin Hamiltonians with respect to Shapovalov form and the key lemma from linear algebra are discussed in Section \ref{shap sec}. In Section \ref{bound sec} we prove our main results, see Theorem~\ref{sgn} and Corollary~\ref{bound thm}. Finally, we display some simple data computed from Corollary~\ref{bound thm} in Section~\ref{sec data}.

\section{Simple Lie algebras}\label{lie sec}
Let $\mathfrak g$ be a simple Lie algebra over $\mathbb C$ with Cartan matrix $A=(a_{i,j})_{i,j=1}^r$, where $r$ is the rank of $\mathfrak g$. Let $D=\operatorname{diag}(d_1,\dots,d_r)$ be the diagonal matrix with positive relatively prime integers $d_i$ such that $DA$ is symmetric.

Let $\mathfrak h\subset\mathfrak g$ be a Cartan subalgebra with the Cartan decomposition $\mathfrak g=\mathfrak n_-\oplus\mathfrak h\oplus\mathfrak n_+$. Fix simple roots $\alpha_1,\dots,\alpha_r$ in $\mathfrak h^*$. Let $\check\alpha_1,\dots,\check\alpha_r\in \mathfrak h$ be the corresponding coroots. Fix a nondegenerate invariant bilinear form $(\ ,\ )$ on $\mathfrak g$ such that $(\check\alpha_i,\check\alpha_j)=a_{i,j}/d_j$. The corresponding bilinear form on $\mathfrak h^*$ is given by $(\alpha_i,\alpha_j)=d_ia_{i,j}$.
We have $\langle \lambda,\check\alpha_i\rangle=2(\lambda,\alpha_i)/(\alpha_i,\alpha_i)$ for $\lambda\in\mathfrak h^*$. In particular, $\langle\alpha_j,\check\alpha_i\rangle=a_{i,j}$. Let $\omega_1,\dots,\omega_r\in\mathfrak h^*$ be the fundamental weights, $\langle \omega_j,\check\alpha_i\rangle=\delta_{i,j}$.

Let $\mathcal P=\{\lambda\in\mathfrak h^*\,|\,\langle \lambda,\check\alpha_i\rangle\in\mathbb Z,\, i=1,\dots,r\}$ and $\mathcal P^+=\{\lambda\in\mathfrak h^*\,|\,\langle \lambda,\check\alpha_i\rangle\in\mathbb Z_{\geqslant 0},\, i=1,\dots,r\}$ be the weight lattice and the cone of dominant integral weights.

Let $e_1,\dots,e_r\in \mathfrak n_+$, $\check\alpha_1,\dots,\check\alpha_r\in \mathfrak h$, $f_1,\dots,f_r\in\mathfrak n_-$ be the Chevalley generators of~$\mathfrak g$.

Given a $\mathfrak g$-module $M$, denote by $(M)^{\mathfrak g}$ the subspace of $\mathfrak g$-invariants in $M$. The subspace $(M)^{\mathfrak g}$ is the multiplicity space of the trivial $\mathfrak g$-module in $M$.

A sequence of nonnegative integers $\xi=(\xi_1,\dots,\xi_k)$ such that $\xi_1\geqslant \xi_2\geqslant \dots\geqslant \xi_k\geqslant 0$ is called {\it a partition with at most $k$ parts}. Set $|\xi|=\sum\limits_{i=1}^k \xi_i$.

For $\lambda\in\mathfrak h^*$, let $V_\lambda$ be the irreducible $\mathfrak g$-module with highest weight $\lambda$. For any $\mathfrak g$-weights $\lambda$ and $\mu$, it is well known that $\dim(V_{\lambda}\otimes V_{\mu})^{\mathfrak g}=\delta_{\lambda,\mu}$ for $\mathfrak g=\mathfrak{so}_{2r+1},\mathfrak{sp}_{2r}$.

For any Lie algebra $\mathfrak g$, denote by $\mathcal U(\mathfrak g)$ the universal enveloping algebra of~$\mathfrak g$.

\section{Characters of the symmetric groups}\label{char sec}
Let $\mathfrak g_N$ be the Lie algebra $\mathfrak{so}_{2r+1}$ if $N=2r$ or the Lie algebra $\mathfrak{sp}_{2r}$ if $N=2r+1$, $r\geqslant 2$. We also set $\mathfrak g_3=\mathfrak{sl}_2$. Let $\mathrm{G}_N$ be the respective classical group with Lie algebra $\mathfrak g_N$.

Let $\mathfrak S_k$ be the symmetric group permuting a set of $k$ elements. In this section we deduce a formula for
characters of a product of the symmetric groups acting on a tensor
product of finite-dimensional irreducible $\mathfrak g_N$-modules.

For each dominant integral $\mathfrak g_N$-weight $\lambda$, denote by $\bar \lambda=(\bar{\lambda}_1,\dots,\bar{\lambda}_r)$ the partition with at most $r$ parts such that
\begin{gather*}
2\langle \lambda,\check\alpha_i\rangle=\bar\lambda_i-\bar\lambda_{i+1},\qquad i=1,\dots,r-1,\qquad\text{and}\qquad \bar{\lambda}_r=\begin{cases}
\langle \lambda,\check\alpha_r\rangle, &\text{if }N=2r,\\
2\langle \lambda,\check\alpha_r\rangle, &\text{if }N=2r+1.
\end{cases}
\end{gather*}

Define an anti-symmetric Laurent polynomial $\Delta_N$ in $x_1,\dots,x_r$ as follows
\begin{gather}\label{eq:delta}
\Delta_N=\det\big(x_i^{N+1-2j}-x_i^{-(N+1-2j)}\big)_{i,j=1}^r.
\end{gather}
We call $\Delta_N$ the \emph{Vandermonde determinant of }$\mathfrak g_N$.

Let $\lambda$ be a dominant integral $\mathfrak g_N$-weight. It is well known that the character of the module~$V_{\lambda}$ is given by
\begin{gather}\label{eq:schur}
\mathcal S_{\lambda}^N(x_1,\dots,x_r)=\operatorname{tr}^{\vphantom1}_{V_{\lambda}}X_N=
\dfrac{\det\big(x_i^{\bar\lambda_j+N+1-2j}-x_i^{-(\bar\lambda_j+N+1-2j)}\big)_{i,j=1}^r}{\Delta_N},
\end{gather}
where $X_N\in\mathrm{G}_N$ is given by
\begin{gather*}
X_N=\begin{cases}
\operatorname{diag}\big(x_1^2,\dots,x_r^2,1,x_r^{-2},\dots,x_1^{-2}\big), &\text{if }N=2r,\\
\operatorname{diag}\big(x_1^2,\dots,x_r^2,x_r^{-2},\dots,x_1^{-2}\big), &\text{if }N=2r+1.
\end{cases}
\end{gather*}
We call $\mathcal S_{\lambda}^N$ the \emph{Schur function of $\mathfrak g_N$ associated with the weight $\lambda$}.

Note that $\mathcal S_\lambda^N$ are symmetric Laurent polynomials in $x_1,\dots,x_r$, $\mathcal S_\lambda^N\in \big(\mathbb C\big[x_1^{\pm 1},\dots,x_r^{\pm 1}\big]\big)^{\mathfrak S_r}$.

Let $\lambda^{(1)},\dots,\lambda^{(s)}$ be a sequence of dominant integral $\mathfrak g_N$-weights and $k_1,\dots,k_s$ a sequence of positive integers. Consider the tensor product of $\mathfrak g_N$-modules
\begin{gather*}
V_{\bm\lambda}=V_{\lambda^{(1)}}^{\otimes k_1}\otimes V_{\lambda^{(2)}}^{\otimes k_2}\otimes\dots\otimes V_{\lambda^{(s)}}^{\otimes k_s}
\end{gather*}
and its decomposition into irreducible $\mathfrak g_N$-submodules
\begin{gather*}
V_{\bm\lambda}=\bigoplus_\mu V_\mu\otimes M_{\bm\lambda,\mu}.
\end{gather*}

By permuting the corresponding tensor factors of $V_{\bm\lambda}$, the product of symmetric groups $\mathfrak S_{\bm k}=\mathfrak S_{k_1} \times \mathfrak S_{k_2} \times\cdots\times \mathfrak S_{k_s}$
acts naturally on $V_{\bm\lambda}$. Note that the $\mathfrak S_{\bm k}$-action commutes with the $\mathfrak g_N$-action, therefore the group $\mathfrak S_{\bm k}$ acts on the multiplicity space $M_{\bm\lambda,\mu}$ for all $\mu$.

For $\sigma=\sigma_1\times\sigma_2\times\dots\times\sigma_s\in \mathfrak S_{\bm k}$, $\sigma_i\in \mathfrak S_{k_i}$. Suppose all $\sigma_i$ are written as a product of disjoint cycles. Denote by $c_i$ the number of cycles in the product representing $\sigma_i$ and $l_{ij}$, $j=1,\dots,c_i$, the lengths of cycles. Note that $l_{i,1}+\dots+l_{i,c_i}=k_i$.

We then consider the value of the character of $\mathfrak S_{\bm k}$ corresponding to the representation $M_{\bm\lambda,\mu}$ on $\sigma$. Let $\chi^{\vphantom1}_{\bm\lambda,\mu}=\operatorname{tr}^{\vphantom1}_{M_{\bm\lambda,\mu}}$.

\begin{Proposition}\label{char prop}
 The character value $\chi^{\vphantom1}_{\bm\lambda,\mu}(\sigma)$ equals the coefficient of the monomial
\begin{gather*}x_1^{\bar\mu_1+N-1}x_2^{\bar\mu_2+N-3}\cdots x_r^{\bar\mu_r+N+1-2r}\end{gather*}
 in the Laurent polynomial
\begin{gather*}
 \Delta_N\cdot \prod_{i=1}^s \prod_{j=1}^{c_i} \mathcal S_{\lambda^{(i)}}^N\big(x_1^{l_{ij}},\dots,x_r^{l_{ij}}\big).
\end{gather*}
\end{Proposition}
\begin{proof}
 The proof of the statement is similar to that of \cite[Proposition 2.1]{MT}.
\end{proof}

\section{Osculating Schubert calculus and self-dual spaces}\label{schubert sec}
Let $N$, $d\in \mathbb Z_{>0}$ be such that $N\leqslant d$. Consider $\mathbb P^1:=\mathbb C\cup\{\infty\}$. Set
\begin{gather*}
{\mathring{\mathbb P}}_n:=\big\{\bm z=(z_1,\dots,z_n)\in (\mathbb{P}^1)^n\,|\, z_i\ne z_j \ \text{for} \ 1\leqslant i<j\leqslant n\big\},\\
\mathbb R{\mathring{\mathbb P}}_n:=\big\{\bm z=(z_1,\dots,z_n)\in {\mathring{\mathbb P}}_n\, |\, z_i\in\mathbb R \ \text{or} \ z_i=\infty, \ \text{for} \ 1\leqslant i\leqslant n\big\}.
\end{gather*}
\subsection{Osculating Schubert calculus}
Let $\mathbb C_d[x]$ be the space of polynomials in $x$ with complex coefficients of degree less than $d$. We have $\dim \mathbb C_d[x] = d$. Let $\operatorname{Gr}(N, d)$ be the Grassmannian of all $N$-dimensional subspaces in $\mathbb C_d[x]$. The Grassmannian $\operatorname{Gr}(N,d)$ is a~smooth projective complex variety of dimension $N(d-N)$.

Let $\mathbb R_d[x]\subset\mathbb C_d[x]$ be the set of polynomials in $x$ with real coefficients of degree less than $d$. Let $\operatorname{Gr}^{\mathbb R}(N,d)\subset \operatorname{Gr}(N,d)$ be the set of subspaces which have a basis consisting of polynomials with all coefficients real. For $X\in\operatorname{Gr}(N,d)$ we have $X\in\operatorname{Gr}^{\mathbb R}(N,d)$ if and only if $\dim_{\mathbb R}(X\cap \mathbb R_d[x])=N$. We call such points $X$ {\it real}.

For a complete flag $\mathcal F=\{0\subset \mathscr F_1\subset \mathscr F_2\subset \dots\subset \mathscr F_d=\mathbb C_{d}[x]\}$ and a partition $\xi=(\xi_1,\dots,\xi_N)$ such that $\xi_1\leqslant d-N$, the Schubert cell $\Omega_{\xi}(\mathcal F)\subset \operatorname{Gr}(N,d)$ is given by
\begin{gather*}
\Omega_{\xi}(\mathcal F)=\big\{X\in\operatorname{Gr}(N,d)\,|\, \dim(X\cap \mathscr F_{d-j-\xi_{N-j}})=N-j,\\
\hphantom{\Omega_{\xi}(\mathcal F)=\big\{}{}\dim(X\cap \mathscr F_{d-j-\xi_{N-j}-1})=N-j-1\big\}.
\end{gather*}
Note that $\operatorname{codim} \Omega_{\xi}(\mathcal F)=|\xi|$.

Let $\mathcal F(\infty)$ be the complete flag given by
\begin{gather}\label{F(inf)}
\mathcal F(\infty)=\big\{0\subset \mathbb C_1[x]\subset\mathbb C_2[x]\subset\dots\subset\mathbb C_d[x]\big\}.
\end{gather}

The subspace $X$ is a point of $\Omega_{\xi}(\mathcal F(\infty))$ if and only if for every $i=1,\dots,N$, it contains a~polynomial of degree $d-i-\xi_{N+1-i}$.

For $z\in\mathbb C$, consider the complete flag
\begin{gather}\label{F(z)}
\mathcal F(z)=\big\{0\subset (x-z)^{d-1}\mathbb C_1[x]\subset(x-z)^{d-2}\mathbb C_2[x]\subset\dots\subset\mathbb C_d[x]\big\}.
\end{gather}

The subspace $X$ is a point of $\Omega_{\xi}(\mathcal F(z))$ if and only if for every $i=1,\dots,N$, it contains a~polynomial with a root at $z$ of order exactly $\xi_i+N-i$.

A point $z\in\mathbb C$ is called a {\it base point} for a subspace $X\subset \mathbb C_d[x]$ if $\varphi(z)=0$ for every $\varphi\in X$.

Let $\bm\xi{=}\big(\xi^{(1)},{\dots},\xi^{(n)}\big)$ be a sequence of partitions with at most $N$ parts and $\bm z{=}(z_1,{\dots},z_n){\in} {\mathring{\mathbb P}}_n$. Set $|\bm\xi|=\sum\limits_{s=1}^n\big|\xi^{(s)}\big|$.

Assuming $|\bm\xi|=N(d-N)$, denote by $\Omega_{\bm\xi,\bm z}$ the intersection of the Schubert cells
\begin{gather}\label{Omega}
\Omega_{\bm\xi,\bm z}= \bigcap_{s=1}^n\Omega_{\xi^{(s)}}(\mathcal F(z_s)).
\end{gather}
Note that due to our assumption, $\Omega_{\bm\xi,\bm z}$ is a finite subset of $\operatorname{Gr}(N,d)$.

Define a sequence of polynomials $\bm T=(T_1,\dots,T_{N})$ by the formulas
\begin{gather*}
T_i(x)=\prod_{s=1}^n (x-z_s)^{\xi_i^{(s)}-\xi_{i+1}^{(s)}},\qquad i=1,\dots,N,
\end{gather*}
where $\xi_{N+1}^{(s)}=0$. Here and in what follows we use the convention that $x-z_s$ is considered as the constant function $1$ if $z_s=\infty$. We say that $\bm T$ is \emph{associated with $\bm\xi$, $\bm z$}.

\subsection{Self-dual spaces}
Let $X\in\operatorname{Gr}(N,d)$ be an $N$-dimensional subspace of polynomials in~$x$. Given a polynomial~$\psi$ in~$x$, denote by $\psi\cdot X$ the space of polynomials of the form $\psi\cdot \varphi$ for all $\varphi\in X$.

Let $X^\vee$ be the $N$-dimensional space of polynomials which are Wronskian determinants of $N-1$ elements of~$X${\samepage
\begin{gather*}
X^\vee=\big\{\det\left(d^{i-1}\varphi_j/dx^{i-1}\right)_{i,j=1}^{N-1}, \, \varphi_j(x)\in X \big\}.
\end{gather*}
The space $X$ is called \emph{self-dual} if $X^\vee= \psi\cdot X$ for some polynomial~$\psi(x)$, see~\cite{MV}.}

Let $\operatorname{sGr}(N,d)$ be the set of all self-dual spaces in $\operatorname{Gr}(N,d)$. We call $\operatorname{sGr}(N,d)$ the \emph{self-dual Grassmannian}. The self-dual Grassmannian $\operatorname{sGr}(N,d)$ is an algebraic subset of $\operatorname{Gr}(N,d)$.

Denote by $\mathrm{s}\Omega_{\bm\xi,\bm z}$ the set of all self-dual spaces in $\Omega_{\bm\xi,\bm z}$
\begin{gather*}
\mathrm{s}\Omega_{\bm\xi,\bm z}=\Omega_{\bm\xi,\bm z}\bigcap \operatorname{sGr}(N,d).
\end{gather*}

Let $\mu$ be a dominant integral $\mathfrak g_N$-weight and $k\in\mathbb Z_{\geqslant 0}$. Define a partition $\mu_{A,k}$ with at most~$N$ parts by the rule: $(\mu_{A,k})_N=k$ and
\begin{gather}\label{A to BC}
(\mu_{A,k})_i-(\mu_{A,k})_{i+1}=\begin{cases}
 \langle \mu,\check\alpha_i\rangle, &\text{if }1\leqslant i\leqslant \big[\frac{N}{2}\big],\\
 \langle \mu,\check\alpha_{N-i}\rangle, &\text{if }\big[\frac{N}{2}\big]< i\leqslant N-1.
\end{cases}
\end{gather}
We call $\mu_{A,k}$ the partition \emph{associated with weight $\mu$ and integer $k$}.

Let $\bm\lambda=\big(\lambda^{(1)},{\dots},\lambda^{(n)}\big)$ be a sequence of dominant integral $\mathfrak g_N$-weights and let $\bm k=(k_1,{\dots},k_n)$ be an $n$-tuple of nonnegative integers. Then denote $\bm\lambda_{A,\bm k}=\big(\lambda_{A,k_1}^{(1)},\dots,\lambda_{A,k_n}^{(n)}\big)$ the sequence of partitions associated with $\lambda^{(s)}$ and $k_s$, $s=1,\dots,n$.

We use the notation $\mu_{A}=\mu_{A,0}$ and $\bm\lambda_{A}=\bm\lambda_{A,(0,\dots,0)}$.

\begin{Lemma}[\cite{LMV}]\label{lem sym weight new} If $\bm\xi$ is a sequence of partitions with at most $N$ parts such that $|\bm\xi|=N(d-N)$ and $\mathrm s\Omega_{\bm\xi,\bm z}$ is nonempty, then $\bm\xi$ has the form $\bm\xi=\bm\lambda_{A,\bm k}$ for a sequence of dominant integral $\mathfrak g_N$-weights $\bm\lambda=\big(\lambda^{(1)},\dots,\lambda^{(n)}\big)$ and a sequence of nonnegative integers $\bm k=(k_1,\dots,k_n)$. The pair $(\bm\lambda,\bm k)$ is uniquely determined by~$\bm \xi$.
\end{Lemma}

In what follows we write $\Omega_{\bm\lambda,\bm z}$, $\Omega_{\bm\lambda,\bm k,\bm z}$, $\mathrm s\Omega_{\bm\lambda,\bm z}$, $\mathrm s\Omega_{\bm\lambda,\bm k,\bm z}$ for $\Omega_{\bm\lambda_A,\bm z}$, $\Omega_{\bm\lambda_{A,\bm k},\bm z}$, $\mathrm s\Omega_{\bm\lambda_A,\bm z}$, $\mathrm s\Omega_{\bm\lambda_{A,\bm k},\bm z}$, respectively.

Note that $|\bm\lambda_{A,\bm k}|=|\bm\lambda_{A}|+N|\bm k|$, where $|\bm k|=k_1+\dots +k_n$. Suppose $|\bm\lambda_A|=N(d-N)$, there exists a bijection between $\Omega_{\bm\lambda,\bm z}$ in $\operatorname{Gr}(N,d)$ and $\Omega_{\bm\lambda,\bm k,\bm z}$ in $\operatorname{Gr}(N,d+|\bm k|)$ given by
\begin{gather}\label{mult}
\Omega_{\bm\lambda,\bm z}\rightarrow \Omega_{\bm\lambda,\bm k,\bm z},\qquad X \mapsto \prod_{s=1}^n(x-z_s)^{k_s}\cdot X.
\end{gather}
Moreover, \eqref{mult} restricts to a bijection between $\mathrm s\Omega_{\bm\lambda,\bm z}$ in $\operatorname{sGr}(N,d)$ and $\mathrm s\Omega_{\bm\lambda,\bm k,\bm z}$ in $\operatorname{sGr}(N,d+|\bm k|)$.

\section{Gaudin model}\label{gaudin sec}
Let $\mathfrak g[t] = \mathfrak g\otimes\mathbb C[t]$ be the Lie algebra of $\mathfrak g$-valued polynomials with the pointwise commutator. We call it the \emph{current algebra} of~$\mathfrak g$. We identify the Lie algebra~$\mathfrak g$ with the subalgebra $\mathfrak g\otimes 1$ of constant polynomials in~$\mathfrak g[t]$.

It is convenient to collect elements of $\mathfrak g[t]$ in generating series of a formal variable $x$. For $g\in \mathfrak g$, set
\begin{gather}\label{currents}
g(x)=\sum_{k=0}^\infty \big(g\otimes t^k\big)x^{-k-1}.
\end{gather}

For each $a\in\mathbb C$, we have the evaluation homomorphism $\mathrm{ev}_a\colon \mathfrak g[t]\to \mathfrak g$ where $\mathrm{ev}_a$ sends $g\otimes t^s$ to $a^sg$ for all $g\in \mathfrak g$ and $s\in \mathbb Z_{\geqslant 0}$. Its restriction to the subalgebra $\mathfrak g\subset \mathfrak g[t]$ is the identity map. For any $\mathfrak g$-module $M$, we denote by $M(a)$ the $\mathfrak g[t]$-module, obtained by pulling $M$ back through the evaluation homomorphism $\mathrm{ev}_a$. The $\mathfrak g[t]$-module $M(a)$ is called an \emph{evaluation module}. The generating series $g(x)$ acts on the evaluation module $M(a)$ by $g/(x-a)$.

The Bethe algebra $\mathcal B$ (the algebra of higher Gaudin Hamiltonians) for a simple Lie algebra~$\mathfrak g$ was described in \cite{FFRe}. The Bethe algebra $\mathcal B$ is a commutative subalgebra of $\mathcal U(\mathfrak g[t])$ which commutes with the subalgebra $\mathcal U(\mathfrak g)\subset \mathcal U(\mathfrak g[t])$. An explicit set of generators of the Bethe algebra in Lie algebras of types B, C, and D was given in \cite{M}.

\begin{Proposition}[\cite{FFRe, M}]\label{prop bethe BC} Let $N>3$.
 There exist elements $F_{ij}\in\mathfrak g_N$, $i,j=1,\dots,N$, and polynomials $B_s(x)$ in $d^k F_{ij}(x)/dx^k$, $s=1,\dots,N$, $k=0,\dots, N$, such that the Bethe algebra $\mathcal B$ of $\mathfrak g_N$ is generated by the coefficients of $B_s(x)$ considered as formal power series in $x^{-1}$.
\end{Proposition}

We denote $M(\infty)$ the $\mathfrak g_N$-module $M$ with the trivial action of the Bethe algebra $\mathcal B$, see \cite{LMV} for more detail.

For a collection of $\mathfrak g_N$-weights $\bm\lambda=\big(\lambda^{(1)},\dots,\lambda^{(n)}\big)$ and $\bm z=(z_1,\dots,z_n)\in \mathring{\mathbb P}_n$, we set
\begin{gather*}
V_{\bm\lambda,\bm z}=\bigotimes_{s=1}^n V_{\lambda^{(s)}}(z_s),
\end{gather*}
considered as a $\mathcal B$-module. We also denote $V_{\bm\lambda}$ the module $V_{\bm\lambda,\bm z}$ considered as a $\mathfrak g_N$-module.

Let $\partial_x$ be the differentiation with respect to $x$. Define a formal differential operator
\begin{gather*}
\mathcal D^{\mathcal B}=\partial_x^N+\sum_{i=1}^NB_i(x)\partial_x^{N-i},
\end{gather*}
where
\begin{gather}\label{eq:def Bij}
B_i(x)=\sum_{j=i}^\infty B_{ij}x^{-j}
\end{gather}
and $B_{ij}\in \mathcal U(\mathfrak g_N[t])$, $j\in\mathbb Z_{\geqslant i}$, $i=1,\dots,N$. The operator $\mathcal D^{\mathcal B}$ is called \emph{the universal operator}.

Let $\bm z=(z_1,\dots,z_n)\in\mathring{\mathbb P}_n$ and let $\bm\lambda=\big(\lambda^{(1)},\dots,\lambda^{(n)}\big)$ be a sequence of dominant integral $\mathfrak g_N$-weights. For every $g\in\mathfrak g_N$, the series $g(x)$ acts on $V_{\bm\lambda,\bm z}$ as a rational function of $x$.

Since the Bethe algebra $\mathcal B$ commutes with $\mathfrak g_N$, $\mathcal B$ acts on the invariant space $(V_{\bm\lambda,\bm z})^{\mathfrak g_N}$. For $b\in\mathcal B$, denote by $b(\bm\lambda,\bm z)\in\operatorname{End}((V_{\bm\lambda,\bm z})^{\mathfrak g_N})$ the corresponding linear operator.

Given a common eigenvector $v\in (V_{\bm\lambda,\bm z})^{\mathfrak g_N}$ of the operators $b(\bm\lambda,\bm z)$, denote by $b(\bm\lambda,\bm z;v)$ the corresponding eigenvalues, and define the scalar differential operator
\begin{gather*}
\mathcal D_v=\partial_x^{N}+\sum_{i=1}^N\sum_{j=i}^\infty B_{ij}(\bm\lambda,\bm z; v)x^{-j}\partial_x^{N-i}.
\end{gather*}

The following theorem connects self-dual spaces in the Grassmannian $\operatorname{Gr}(N,d)$ with the Gaudin model associated to $\mathfrak g_N$.

\begin{Theorem}[\cite{LMV}]\label{bi rep sgr}
 Let $N>3$. There exists a choice of generators $B_i(x)$ of the Bethe algebra~$\mathcal B$, such that for any sequence of dominant integral $\mathfrak g_N$-weights $\bm\lambda=\big(\lambda^{(1)},\dots,\lambda^{(n)}\big)$, any $\bm z\in \mathring{\mathbb P}_n$, and any $\mathcal B$-eigenvector $v\in (V_{\bm\lambda,\bm z})^{\mathfrak g_N}$, we have
\begin{gather*}
 \operatorname{Ker} \big((T_1\cdots T_{N})^{1/2}\cdot \mathcal D_v\cdot(T_1\cdots T_{N})^{-1/2}\big)\in \mathrm s\Omega_{\bm\lambda,\bm z},
\end{gather*}
 where $\bm T=(T_1,\dots,T_N)$ is associated with $\bm\lambda_A$, $\bm z$.

 Moreover, if $|\bm\lambda_A|=N(d-N)$, then this defines a bijection between the joint eigenvalues of~$\mathcal B$ on $(V_{\bm\lambda,\bm z})^{\mathfrak g_N}$ and $\mathrm s\Omega_{\bm\lambda,\bm z}\subset\operatorname{Gr}(N,d)$.
\end{Theorem}

\section{Shapovalov form and the key lemma}\label{shap sec}
\subsection{Shapovalov form}
Define the anti-involution $\varpi\colon \mathfrak g_N \to \mathfrak g_N $ sending $e_1, \dots,e_r$, $\check\alpha_1,\dots,\check\alpha_r$, $f_1,\dots, f_r$ to $f_1,\dots, f_r$, $\check\alpha_1,\dots,\check\alpha_r$, $e_1, \dots,e_r$, respectively.

For any dominant integral $\mathfrak g_N$-weight $\lambda$, the irreducible $\mathfrak g_N$-module $V_\lambda$ admits a positive definite Hermitian form $(\cdot,\cdot)_\lambda$ such that $(gv,w)_\lambda= (v,\overline{\varpi(g)}w)_\lambda$ for any $v,w\in V_\lambda$ and $g\in\mathfrak g_N$. Such a form is unique up to multiplication by a positive real number. We call this form the \emph{Shapovalov form}.

Let $\bm\lambda=\big(\lambda^{(1)},\dots,\lambda^{(n)}\big)$ be a sequence of dominant integral $\mathfrak g_N$-weights. We define the positive definite Hermitian form $(\cdot,\cdot)_{\bm\lambda}$ on the tensor product $V_{\bm\lambda}$ as the product of Shapovalov forms on the tensor factors. The form $(\cdot,\cdot)_{\bm\lambda}$ induces a positive definite Hermitian form $(\cdot|\cdot)_{\bm\lambda}$ on $(V_{\bm\lambda,\bm z})^{\mathfrak g_N}$.

\begin{Proposition} \label{sym bethe}For any $i=1,\dots,N$, $j\in\mathbb Z_{\geqslant i}$, and any $v,w\in (V_{\bm\lambda,\bm z})^{\mathfrak g_N}$, we have
\begin{gather*}
 \big(B_{ij}(\bm\lambda,\bm z)v|w\big)_{\bm\lambda}= \big(v|B_{ij}(\bm\lambda,\bar{\bm z})w\big)_{\bm\lambda},
\end{gather*}
 where $B_{ij}$ are given by \eqref{eq:def Bij}, $\bar{\bm z}=(\bar z_1,\dots,\bar z_n)$ and the bar stands for the complex conjugation.
\end{Proposition}
\begin{proof}We prove the proposition in Section~\ref{proof}.
\end{proof}

If $\bm z\in\mathbb R{\mathring{\mathbb P}}_n$, then $B_{ij}(\bm\lambda,\bm z)$ are self-adjoint with respect to the Shapovalov form. Therefore all $B_{ij}(\bm\lambda,\bm z)$ are simultaneously diagonalizable and all eigenvalues of $B_{ij}(\bm\lambda,\bm z)$ are real.

The following statement is also known.

\begin{Theorem}[\cite{R}]\label{lem BC completeness}
 For generic $\bm z\in {\mathring{\mathbb P}}_n$, the action of the Bethe algebra $\mathcal B$ on $(V_{\bm\lambda,\bm z})^{\mathfrak g_N}$ is diagonalizable and has simple spectrum. In particular, this statement holds for any sequence $\bm z\in \mathbb R{\mathring{\mathbb P}}_n$.
\end{Theorem}

If some of the partitions $\lambda^{(1)},\dots,\lambda^{(n)}$ coincide, the operators $b(\bm\lambda,\bm z)$ admit additional symmetry. Assume that $\lambda^{(i)}=\lambda^{(i+1)}$ for some $i$. Let $P_i\in\operatorname{End}(V_{\bm\lambda})$ be the flip of the $i$-th and $(i+1)$-st tensor factors and $\tilde{{\bm z}}^{(i)}=(z_1,\dots,z_{i-1},z_{i+1},z_i,z_{i+2},\dots,z_n)$.

\begin{Lemma}\label{perm bethe} For any $b\in\mathcal B$, we have $P_i b(\bm\lambda,\bm z)P_i=b\big(\bm\lambda,\tilde{{\bm z}}^{(i)}\big)$.
\end{Lemma}

\subsection{Self-adjoint operators with respect to indefinite Hermitian form}\label{lin alg sec}
In this section we recall the key lemma from linear algebra, see \cite{P}.

Given a finite-dimensional vector space $M$, a linear operator $\mathfrak T\in\operatorname{End}(M)$, and a number $\alpha\in\mathbb C$, let $M_{\mathfrak T}(\alpha)=\operatorname{ker} (\mathfrak T-\alpha)^{\dim M}$. When $M_{\mathfrak T}(\alpha)$ is not trivial, it is the subspace of generalized eigenvectors of $\mathfrak T$ with eigenvalue $\alpha$.

\begin{Lemma}[\cite{P}] \label{linalg2} Let $M$ be a complex finite-dimensional vector space with a nondegenerate Hermitian form of signature $\kappa$, and let ${\mathcal A}\subset\operatorname{End}(M)$ be a commutative subalgebra over $\mathbb R$, whose elements are self-adjoint operators. Let $R=\bigcap_{\mathfrak T\in\mathcal A}\bigoplus_{\alpha\in\mathbb R}M_{\mathfrak T}(\alpha)$. Then the restriction of the Hermitian form on $R$ is nondegenerate and has signature $\kappa$. In particular, $\dim R\geqslant| \kappa|$.
\end{Lemma}

\subsection{Proof of Proposition~\ref{sym bethe}}\label{proof}
In this section, we give the proof of Proposition~\ref{sym bethe}. We follow the convention of \cite{MM}. We only introduce the necessary notation and refer the reader to \cite[Section~5]{M} and \cite[Section~3]{MM} for more detail.

\begin{proof}[Proof of Proposition \ref{sym bethe}] We prove it for the case $N=2r$ first.

Let $E_{ij}$ with $i,j=1,\dots,2r+1$ be the standard basis of~$\mathfrak{gl}_{2r+1}$. The Lie subalgebra of~$\mathfrak{gl}_{2r+1}$ generated by the elements $F_{ij}=E_{ij}-E_{2r+2-j,2r+2-i}$ is isomorphic to the Lie algebra \smash{$\mathfrak{so}_{2r+1}=\mathfrak g_N$}. With this isomorphism, the anti-involution $\varpi\colon \mathfrak g_N\to\mathfrak g_N$ is realized by taking transposition, $F_{ij}\mapsto F_{ji}$. To be consistent with the notation in \cite{MM}, we write $\mathfrak g$ for $\mathfrak g_N$. The number~$N$ in~\cite{MM} is~$2r+1$ in our notation.

 We write $F_{ij}[s]$ for $F_{ij}\otimes t^{s}$ in the loop algebra $\mathfrak g\big[t,t^{-1}\big]$. Consider the affine Lie algebra $\widehat{\mathfrak g}=\mathfrak g\big[t,t^{-1}\big]\oplus \mathbb C K$, which is the central extension of the loop algebra $\mathfrak g\big[t,t^{-1}\big]$, where the element $K$ is central in $\widehat{\mathfrak g}$ and
\begin{gather*}
 [g_1[k],g_2[l]]=[g_1,g_2][k+l]+k\delta_{k,-l}(g_1,g_2)K,\qquad g_1,g_2\in\mathfrak g,\qquad k,l\in \mathbb Z.
\end{gather*}
 Consider the extended affine Lie algebra $\widehat{\mathfrak g}\oplus \mathbb C\tau=\mathfrak g\big[t,t^{-1}\big]\oplus \mathbb C K\oplus \mathbb C\tau$, where $\tau$ satisfies
\begin{gather*}
 [\tau,F_{ij}[s]]=-sF_{ij}[s-1],\qquad [\tau,K]=0,\qquad s\in \mathbb Z.
\end{gather*}
Set $\mathcal U=\mathcal U(\widehat{\mathfrak g}\oplus \mathbb C\tau)$ and fix $m\in \{1,\dots,N\}$. Introduce the element $F[s]_a$ of the algebra $\big(\operatorname{End} \big(\mathbb C^{2r+1}\big)\big)^{\otimes m}\otimes \mathcal U$, see \cite[equation~(3.5)]{MM}, by
\begin{gather*}
 F[s]_a=\sum_{i,j=1}^{2r+1} 1^{\otimes (a-1)}\otimes e_{ij}\otimes 1^{\otimes (m-a)}\otimes F_{ij}[s],
\end{gather*}
 where $e_{ij}\in \operatorname{End}\big(\mathbb C^{2r+1}\big)$ denote the standard matrix units. The map $\varpi$ induces an anti-involution
\begin{gather*}
 \varpi\colon \ \mathcal U\big(t^{-1}\mathfrak g\big[t^{-1}\big]\big)\to \mathcal U\big(t^{-1}\mathfrak g\big[t^{-1}\big]\big),\qquad F_{ij}[s]\mapsto F_{ji}[s],\qquad s\in\mathbb Z_{\leqslant -1}.
\end{gather*}

 For $1\leqslant a < b \leqslant m$, consider the operators $P_{ab}$ and $Q_{ab}$ in $\big(\operatorname{End}\big(\mathbb C^{2r+1}\big)\big)^{\otimes m}$ defined as follows
\begin{gather*}
 P_{ab}=\sum_{i,j=1}^{2r+1}1^{\otimes (a-1)}\otimes e_{ij}\otimes 1^{\otimes (b-a-1)}\otimes e_{ji}\otimes 1^{\otimes (m-b)},\\
 Q_{ab}=\sum_{i,j=1}^{2r+1}1^{\otimes (a-1)}\otimes e_{ij}\otimes 1^{\otimes (b-a-1)}\otimes e_{2r+2-i,2r+2-j}\otimes 1^{\otimes (m-b)}.
\end{gather*}
 Set
\begin{gather*}
 S^{(m)}=\frac{1}{m!}\prod_{1\leqslant a<b\leqslant m}\left(1+\frac{P_{ab}}{b-a}-\frac{2Q_{ab}}{2r+2b-2a-1}\right),
\end{gather*}
where the product is taken in the lexicographic order on the pairs $(a,b)$. The element $S^{(m)}$ is the symmetrizer of the Brauer algebra acting on $\big(\mathbb C^{2r+1}\big)^{\otimes m}$. In particular, for any $1\leqslant a<b\leqslant m$, the operator $S^{(m)}$ satisfies
\begin{gather*}
 S^{(m)}Q_{ab}=Q_{ab}S^{(m)}=0,\qquad S^{(m)}P_{ab}=P_{ab}S^{(m)}=S^{(m)}.
\end{gather*}

Replacing $\tau$ with $\partial_x$ and $F_{ij}[-\ell-1]$ with $-\partial_x^{\ell}F_{ij}(x)/\ell!$, where $F_{ij}(x)$ is defined in \eqref{currents}, for $\ell\in \mathbb Z_{\geqslant 0}$, in the element
\begin{gather*}
 \frac{2r+m-1}{2r+2m-1}\operatorname{tr}S^{(m)}(\tau+F[-1]_1)\cdots(\tau+F[-1]_m),
\end{gather*}
see \cite[formula~(3.26)]{MM}, where the trace is taken on all $m$ copies of $\operatorname{End}\big(\mathbb C^{2r+1}\big)$, we get a~differential operator
\begin{gather*}
 \vartheta_{m0}(x)\partial_x^m+\vartheta_{m1}(x)\partial_x^{m-1}+\dots+\vartheta_{mm}(x),
\end{gather*}
where $\vartheta_{mi}(x)$ is a formal power series in $x^{-1}$ with coefficients in~$\mathcal U(\mathfrak g[t])$. The Bethe subalgebra~$\mathcal B$ of $\mathcal U(\mathfrak g[t])$ is generated by the coefficients of $\vartheta_{mi}(x)$, $m=1,\dots,N$, $i=0,\dots,m$, see~\cite[Section~5]{M}.

Therefore, to prove the proposition, it suffices to show that the element
\begin{gather}\label{diff-gen}
 \frac{2r+m-1}{2r+2m-1}\operatorname{tr}S^{(m)}(\tau+F[-1]_1)\cdots(\tau+F[-1]_m),
\end{gather}
 is stable under the anti-involution $\varpi$. Here $\varpi$ maps $\tau$ to $\tau$.

 Applying transposition on $a$-th and $b$-th components to the commutator relation
\begin{gather*}
 F[k]_aF[l]_b-F[l]_bF[k]_a=(P_{ab}-Q_{ab})F[k+l]_b-F[k+l]_b(P_{ab}-Q_{ab}),
\end{gather*}
 see the proof of \cite[Lemma 3.6]{MM}, we get
\begin{gather*}
 F^{\top}[k]_aF^{\top}[l]_b-F^{\top}[l]_bF^{\top}[k]_a=F^{\top}[k+l]_b(P_{ab}-Q_{ab})-(P_{ab}-Q_{ab})F^{\top}[k+l]_b,
\end{gather*}for all $1\leqslant a<b\leqslant m$. Here $\top$ stands for transpose, explicitly,
\begin{gather*}
 F^\top[s]_a=\sum_{i,j=1}^{2r+1} 1^{\otimes (a-1)}\otimes e_{ij}\otimes 1^{\otimes (m-a)}\otimes F_{ji}[s].
\end{gather*}
 Thus one can use the same argument as in the proof of \cite[Lemma 3.2]{MM} to show that the image of \eqref{diff-gen} under the anti-involution $\varpi$ equals
\begin{gather}\label{diff-oper}
 \frac{2r+m-1}{2r+2m-1}\operatorname{tr}S^{(m)}\big(\tau+F^{\top}[-1]_1\big)\cdots\big(\tau+F^{\top}[-1]_m\big).
\end{gather}
 By applying the simultaneous transposition $e_{ij}\to e_{ji}$ to all $m$ copies of $\operatorname{End}\big(\mathbb C^{2r+1}\big)$ we conclude that~\eqref{diff-oper} coincides with~\eqref{diff-gen} because the transformation takes each factor $\tau+F^{\top}[-1]_a$ to $\tau+F[-1]_a$ whereas the symmetrizer~$S^{(m)}$ stays invariant. Hence we complete the proof of Proposition~\ref{sym bethe} for the case $N=2r$.

 The case $N=2r+1$ is proved similarly, see for example \cite[Lemma~3.9]{MM}.
\end{proof}

\section{The lower bound}\label{bound sec}
In this section we prove our main results -- the lower bound for the number of real self-dual spaces in $\Omega_{\bm\lambda,\bm z}$, see Theorem~\ref{sgn} and Corollary~\ref{bound thm}.

Recall the notation from Section \ref{schubert sec}. For positive integers $N$, $d$ such that $d\geqslant N$ we consider the Grassmannian $\operatorname{Gr}(N,d)$ of $N$-dimensional planes in the space~$\mathbb C_d[x]$ of polynomials of degree less than~$d$. A point $X\in\operatorname{Gr}(N,d)$ is called real if it has a basis consisting of polynomials with all coefficients real.

\subsection[The general case $N\geqslant 4$]{The general case $\boldsymbol{N\geqslant 4}$}\label{sec gen}
Let us first consider the case $N\geqslant 4$.

Let $\bm\lambda=\big(\lambda^{(1)},\dots,\lambda^{(n)}\big)$ be a sequence of dominant integral $\mathfrak g_N$-weights, $\bm k=(k_1,\dots,k_n)$ an $n$-tuple of nonnegative integers, and $\bm z=(z_1,\dots,z_n)\in\mathring{\mathbb P}_n$. Suppose that $|\bm\lambda_{A,\bm k}|=N(d-N)$. Denote by $d(\bm\lambda,\bm k,\bm z)$ the number of real points counted with multiplicities in $\mathrm s\Omega_{\bm\lambda,\bm k,\bm z}\subset\operatorname{Gr}(N,d)$.

Clearly, $d(\bm\lambda,\bm k,\bm z)=0$ unless the set $\{z_1,\dots,z_n\}$ is invariant under the complex conjugation and $\big(\lambda^{(i)},k_i\big) =\big(\lambda^{(j)},k_j\big)$ whenever $z_i=\bar z_j$. In particular, the polynomial $\prod\limits_{s=1}^n(x-z_s)^{k_s}$ has only real coefficients. It follows from \eqref{mult} that the number of real points in $\mathrm s\Omega_{\bm\lambda,\bm k,\bm z}\subset\operatorname{Gr}(N,d)$ is equal to that of $\mathrm s\Omega_{\bm\lambda,\bm z}\subset\operatorname{Gr}(N,d-|\bm k|)$. From now on, we shall only consider the case that $\bm k=(0,\dots,0)$. We simply write $d(\bm\lambda,\bm z)$ for $d(\bm\lambda,\bm k,\bm z)$ if $\bm k=(0,\dots,0)$.

Let $\bm T=(T_1,\dots,T_N)$ be associated with $\bm\lambda_{A,\bm k}$, $\bm z$. Note that if $\bm z$, $\bm\lambda$, $\bm k$ is invariant under conjugation, then the polynomial $T_1\cdots T_N$ also has only real coefficients.

In what follows we denote by $c$ the number of complex conjugate pairs in the set $\{z_1,\dots,z_n\}$ and without loss of generality assume that $z_1=\bar z_2,\dots,z_{2c-1}= \bar z_{2c}$ while $z_{2c+1},\dots,z_n$ are real (one of them can be infinity). We will also always assume that $\lambda^{(1)}=\lambda^{(2)},\dots,\lambda^{(2c-1)} =\lambda^{(2c)}$.

Recall that for any $\bm\lambda$ and generic $\bm z\in\mathring{\mathbb P}_n$, all points of $\Omega_{\bm\lambda,\bm z}$ are multiplicity-free. The same also holds true with $\bm\lambda$ imposed above for any $c$.

Consider the decomposition of $V_{\bm\lambda}$ into irreducible $\mathfrak g_N$-submodules
\begin{gather*}
V_{\bm\lambda}=\bigoplus_\mu V_\mu\otimes M_{\bm\lambda, \mu}.
\end{gather*}
Then $M_{\bm\lambda,0}= (V_{\bm\lambda})^{\mathfrak g_N}$. Since $\lambda^{(2i-1)}=\lambda^{(2i)}$ for $i=1,\dots,c$, the flip $P_{2i-1}$ of the $(2i-1)$-st and $2i$-th tensor factors of $V_{\bm\lambda}$ commutes with the $\mathfrak g_N$-action and thus acts on~$(V_{\bm\lambda})^{\mathfrak g_N}$. Denote by $P_{\bm\lambda,c}\in\operatorname{End}((V_{\bm\lambda})^{\mathfrak g_N})$ the action of the product $P_1P_3\cdots P_{2c-1}$ on $(V_{\bm\lambda})^{\mathfrak g_N}$.

The operator $P_{\bm\lambda,c}$ is self-adjoint relative to the Hermitian form $(\cdot|\cdot)_{\bm\lambda}$ on $(V_{\bm\lambda})^{\mathfrak g_N}$ given in Section~\ref{shap sec}. Define a new Hermitian form $(\cdot,\cdot)_{\bm\lambda,c}$ on $(V_{\bm\lambda})^{\mathfrak g_N}$ by the rule: for any $v,w\in(V_{\bm\lambda})^{\mathfrak g_N}$
\begin{gather*}
(v,w)_{\bm\lambda,c} = (P_{\bm\lambda,c} v|w)_{\bm\lambda} .
\end{gather*}
Denote by $q(\bm\lambda,c)$ the signature of the form $(\cdot,\cdot)_{\bm\lambda, c}$.

\begin{Proposition} \label{charsign}The signature $q(\bm\lambda,c)$ equals the coefficient of the monomial
\begin{gather*}
 x_1^{N-1}x_2^{N-3}\cdots x_r^{N+1-2r},
\end{gather*}
in the Laurent polynomial
\begin{gather*}
 \Delta_N\cdot\prod_{i=1}^c \mathcal S_{\lambda^{(2i)}}^N\big(x_1^2,\dots,x_r^2\big) \prod_{j=2c+1}^n \mathcal S_{\lambda^{(j)}}^N(x_1,\dots,x_r) .
\end{gather*}
Here $\Delta_N$ and $\mathcal S_{\lambda^{(s)}}^N$ are given by \eqref{eq:delta} and \eqref{eq:schur}, respectively.
\end{Proposition}
\begin{proof}Since $P_{\bm\lambda,c}^2 =1$ and $M_{\bm\lambda,0}= (V_{\bm\lambda})^{\mathfrak g_N}$, we have $q(\bm\lambda,c)=\operatorname{tr}^{\vphantom1}_{M_{\bm\lambda,0}} P_{\bm\lambda,c}$, and the claim follows from Prop\-osition~\ref{char prop}.
\end{proof}

\begin{Theorem} \label{sgn}
 The number $d(\bm\lambda,\bm z)$ of real self-dual spaces in $\Omega_{\bm\lambda,\bm z}$ is no less than $|q(\bm\lambda,c)|$.
\end{Theorem}
\begin{proof}
 Our proof is parallel to that of \cite[Theorem 7.2]{MT}.

By Proposition \ref{sym bethe} and Lemma \ref{perm bethe}, the operators $B_{ij}(\bm\lambda,\bm z)\in\operatorname{End}((V_{\bm{\lambda}})^{\mathfrak g_N})$ are self-adjoint relative to the form $(\cdot,\cdot)_{\bm\lambda,c}$. By Lemma \ref{linalg2},
\begin{gather*}
 \dim\bigg( \bigcap_{i,j} \bigoplus_{\alpha\in\mathbb R} \big((V_{\bm\lambda})^{\mathfrak g_N}\big)_{B_{ij}(\bm\lambda,\bm z)}(\alpha)\bigg)\geqslant|q(\bm\lambda,c)| .
\end{gather*}
By Theorem \ref{lem BC completeness}, for any $\bm\lambda$ and generic $\bm z\in\mathring{\mathbb P}_n$ the operators $B_{ij}(\bm\lambda,\bm z)$ are diagonalizable and the action of the Bethe algebra $\mathcal B$ on $(V_{\bm\lambda})^{\mathfrak g_N}$ has simple spectrum. The same also holds true with~$\bm\lambda$ imposed above for any~$c$. Thus for generic~$\bm z$, the operators $B_{ij}(\bm\lambda,\bm z)$ have at least $|q(\bm\lambda,c)|$ common eigenvectors with distinct real eigenvalues, which provides $|q(\bm\lambda,c)|$ distinct real points in $\mathrm s\Omega_{\bm\lambda,\bm z}$ by Theorem~\ref{bi rep sgr}. Hence, $d(\bm\lambda,\bm z)\geqslant|q(\bm\lambda,c)|$ for generic $\bm z$, and therefore, for any $\bm z$, due to counting with multiplicities.
\end{proof}

\begin{Remark}\label{rem par}If $\dim (V_{\bm\lambda})^{\mathfrak g_N}$ is odd, it follows from Theorem \ref{sgn} by counting parity that
\begin{gather*}
 d(\bm\lambda,\bm z)\geqslant |q(\bm\lambda,c)|\geqslant 1.
\end{gather*}
Therefore, there exists at least one real point in $\mathrm s\Omega_{\bm\lambda,\bm z}$. In particular, if $\dim (V_{\bm\lambda})^{\mathfrak g_N}=1$, then the only point in $\mathrm s\Omega_{\bm\lambda,\bm z}$ is always real.
\end{Remark}

The following corollary of Proposition \ref{charsign} and Theorem \ref{sgn} is our main result.
\begin{Corollary} \label{bound thm}
 The number $d(\bm\lambda,\bm z)$ of real self-dual spaces in $\Omega_{\bm\lambda,\bm z}$ $($real points in $\mathrm s\Omega_{\bm\lambda,\bm z})$ is no less than $|a(\bm\lambda,c)|$, where $a(\bm\lambda, c)$ is the coefficient of the monomial $x_1^{N-1}x_2^{N-3}\cdots x_r^{N+1-2r}$
 in the Laurent polynomial
\begin{gather*}
 \Delta_N\cdot\prod_{i=1}^c \mathcal S_{\lambda^{(2i)}}^N\big(x_1^2,\dots,x_r^2\big)
 \prod_{j=2c+1}^n \mathcal S_{\lambda^{(j)}}^N(x_1,\dots,x_r) .
\end{gather*}
 Here $\Delta_N$ is the Vandermonde determinant of $\mathfrak g_N$ and $\mathcal S_{\lambda^{(s)}}^N$ is the Schur function of $\mathfrak g_N$ associated with $\lambda^{(s)}$, $s=1,\dots, n$, see~\eqref{eq:delta} and~\eqref{eq:schur}.
\end{Corollary}

\begin{Remark}
 Recall that the total number of points (counted with multiplicities) in $\mathrm s\Omega_{\bm\lambda,\bm z}$ equals
 $\dim(V_{\bm\lambda})^{\mathfrak g_N}=q(\bm\lambda,0)$. Hence if $\bm z\in\mathbb R\mathring{\mathbb P}_n$, Theorem \ref{sgn} claims that all points in
 $\mathrm s\Omega_{\bm\lambda,\bm z}$ are real. It is proved in \cite[Corollary 6.3]{MTV4} that for
 $\bm z\in\mathbb R\mathring{\mathbb P}_n$ all points in $\Omega_{\bm\lambda,\bm z}$ are real and
 multiplicity-free, so are the points in $\mathrm s\Omega_{\bm\lambda,\bm z}$.
\end{Remark}

\subsection[The case $N=2,3$]{The case $\boldsymbol{N=2,3}$}\label{sec N=2,3}
Now let us consider the case $N=2, 3$. Note that $\operatorname{sGr}(2,d)=\operatorname{Gr}(2,d)$, this case is the usual Grassmannian, which has already been discussed in \cite{MT}.

Let $N=3$ and $\mathfrak g_3=\mathfrak{sl}_2$. It suffices for us to consider the case that points in $\mathrm s\Omega_{\bm\lambda,\bm z}$ have no base points, see the beginning of Section \ref{sec gen} for more detail. We shall consider $\operatorname{sGr}(3,2d-1)$ instead of $\operatorname{sGr}(3,d)$, see \cite[Section~4.6]{LMV}. We identify the dominant integral $\mathfrak{sl}_2$-weights with nonnegative integers. Let $\bm\lambda=\big(\lambda^{(1)},\dots,\lambda^{(n)}\big)$ be a sequence of nonnegative integers and $\bm z=(z_1,\dots,z_n)\in {\mathring{\mathbb P}}_{n}$. Then $\bm\lambda_{A}$ has coordinates
\begin{gather*}
\lambda_{A}^{(s)}=\big(2\lambda^{(s)},\lambda^{(s)},0\big),\qquad s=1,\dots,n.
\end{gather*}
We also assume $|\bm\lambda_A|=6(d-2)$.

Recall from \cite[Theorem 4.19]{LMV}, if $X\in \mathrm s\Omega_{\bm\lambda,\bm z}$, then there exist monic polynomials $\varphi$ and $\psi$ such that $\varphi^2$, $\varphi\psi$, $\psi^2$ form a basis of $X$. Denote by $\sqrt{X}$ the space of polynomials spanned by $\varphi$ and $\psi$. Let $\xi^{(i)}$ be the partitions with at most two parts defined by $\big(\lambda^{(i)},0\big)$, $i=1,\dots,n$. Set $\bm \xi=\big(\xi^{(1)},\dots,\xi^{(n)}\big)$, then $|\bm \xi|=2(d-2)$. It follows from the proof of \cite[Theorem~4.19]{LMV} that $\sqrt{X}\in \Omega_{\bm \xi,\bm z}\subset \operatorname{Gr}(2,d)$. The map $\Omega_{\bm \xi,\bm z}\to \mathrm s\Omega_{\bm\lambda,\bm z}$ given by $\sqrt{X}\mapsto X$ is bijective.

\begin{Lemma}\label{red usual grass}The self-dual space $X$ is real if and only if $\sqrt{X}$ is real.
\end{Lemma}
\begin{proof}It is obvious that $X$ is real if $\sqrt{X}$ is real.

Conversely, if $X$ is real, then there exist complex numbers $a_i$, $b_i$, $c_i$, $i=1,2,3$, such that
\begin{gather*}
a_i\varphi^2+b_i\varphi\psi+c_i\psi^2,\qquad i=1,2,3,
\end{gather*}are real polynomials and form a basis of $X$. Without loss of generality, we assume $\deg \varphi<\deg \psi$. Since $\deg \varphi<\deg \psi$, we have $c_i\in\mathbb R$, $i=1,2,3$. At least one of $c_i$ is nonzero. We assume $c_3\ne 0$. By subtracting a proper real multiple of $a_3\varphi^2+b_3\varphi\psi+c_3\psi^2$, we assume further $c_1=c_2=0$. Continuing with the previous step, we assume that $b_1=0$, $b_2\ne 0$, $a_1\ne 0$ and hence obtain that $a_1,b_2,c_3\in\mathbb R$. Then $a_1\varphi^2$ is a real polynomial, so is $\varphi$. Therefore, $a_2\varphi+b_2\psi$ is also a real polynomial, which implies that the space of polynomials $\sqrt{X}$ is also real.
\end{proof}

Because of Lemma \ref{red usual grass}, the case $N=3$ is reduced to the lower bound for the number of real solutions to osculating Schubert problems of $\operatorname{Gr}(2,d)$, see \cite{MT}. Moreover,
Corollary~\ref{bound thm} also applies for this case by putting $N=3$, $r=1$, and $\mathfrak g_N=\mathfrak{sl}_2$.

\section[Some data for small $N$]{Some data for small $\boldsymbol{N}$}\label{sec data}
In this section, we give some data obtained from Corollary~\ref{bound thm} when $N$ is small. Since the cases $N=2,3$ reduce to the cases of~\cite{MT}, we start with $N=4$.

We always assume that $\bm\lambda$, $\bm k$, $\bm z$ are invariant under conjugation. By Remark~\ref{rem par}, we shall only consider the cases that $\dim (V_{\bm\lambda})^{\mathfrak g_N}\geqslant 2$. We also exclude the cases that $\bm z\in\mathbb R\mathring{\mathbb P}_n$. In particular, the cases that all pairs $\big(\lambda^{(s)},k_s\big)$, $s=1,\dots,n$, are different.

We write the highest weights in terms of fundamental weights, for example $(1,0,0,1)=\omega_1+\omega_4$. We also write $\big(\lambda^{(1)}\big)_{k_1},\dots,\big(\lambda^{(n)}\big)_{k_n}$ for $(\bm\lambda,\bm k)$ and simply write $\lambda^{(s)}$ for $(\lambda^{(s)})_{0}$. We use $\big(\lambda_1^{(s)},\lambda_2^{(s)}\big)_{k_s}^{\otimes m}$ to indicate that the pair $\big(\big(\lambda_1^{(s)},\lambda_2^{(s)}\big),k_s\big)$ appears in $(\bm\lambda,\bm k)$ exactly $m$ times. For instance, $(0,1)_1,(0,1)^{\otimes 3}$ represents the pair $(\bm\lambda,\bm k)$ where $\bm\lambda=((0,1),(0,1),(0,1),(0,1))$ and $\bm k=(1,0,0,0)$.

\subsection[The case $N=4,5$]{The case $\boldsymbol{N=4,5}$}\label{sec:n=4}
For each $\mathfrak g_4$-weight $\lambda=(\lambda_1,\lambda_2)$, denote by $\lambda_C$ the $\mathfrak g_5$-weight $(\lambda_2,\lambda_1)$. Note that $\mathfrak g_4=\mathfrak{so}_5$ is isomorphic to $\mathfrak g_5=\mathfrak{sp}_4$, the lower bound obtained from the ramification data $\bm\lambda=\big(\lambda^{(1)},\dots,\lambda^{(n)}\big)$ and $\bm k=(k_1,\dots,k_n)$ of $\mathfrak g_4$ is the same as that obtained from the ramification data $\bm{\lambda_C}=\big(\lambda_C^{(1)},\dots,\lambda_C^{(n)}\big)$ and $\bm k=(k_1,\dots,k_n)$ of~$\mathfrak g_5$.

\begin{table}[h!] \centering
 \begin{tabular}{|c|c|c|c|c|}
 \hline
 ramification data & dimension & $c=1$ & $c=2$ & $c=3$\\
 \hline
 $(0,1)^{\otimes 6}$
 & 14 & 2 & 2 & 6\\
 \hline
 $(1,0)^{\otimes 3},(0,1)^{\otimes 2}$ & 4 & 0,2 & 2 & \\
 \hline
 $(1,0)^{\otimes 3},(1,0)_1$ & 3 & 1 & & \\
 \hline
 $(1,0)^{\otimes 4},(0,0)_1$ & 3 & 1 & 3 & \\
 \hline
 $(0,2),(0,1)^{\otimes 4}$ & 6 & 0 & 2 & \\
 \hline
 $(0,0)_1,(0,1)^{\otimes 4}$ & 3 & 1 & 3 & \\
 \hline
 $(1,0),(0,1)^{\otimes 4}$ & 5 & 1 & 1 & \\
 \hline
 $(1,1),(0,1)^{\otimes 3}$ & 2 & 0 & & \\
 \hline
 $(0,1)_1,(0,1)^{\otimes 3}$ & 3 & 1 & & \\
 \hline
 $(0,2)^{\otimes 2},(0,1)^{\otimes 2}$ & 3 & 1 & 3 & \\
 \hline
 $(1,0)^{\otimes 2},(0,1)^{\otimes 2}$ & 2 & 0 & 2 & \\
 \hline
 $(0,2),(1,0),(0,1)^{\otimes 2}$ & 2 & 0 & & \\
 \hline
 $(1,0)_1,(1,0),(0,1)^{\otimes 2}$ & 2 & 0 & &\\
 \hline
 $(1,0)^{\otimes 2},(0,1)_1,(0,1)$ & 2 & 0 & &\\
 \hline
 $(1,1),(1,0)^{\otimes 2},(0,1)$ & 2 & 0 & &\\
 \hline
 \end{tabular}
 \caption{The case $N=4,5$.} \label{table:1}
\end{table}

In Table \ref{table:1}, we give lower bounds for the cases from $\operatorname{Gr}(4,7)$ and $\operatorname{Gr}(5,10)$. By the observation above, we transform the case from $\operatorname{Gr}(5,10)$ to its counter part in $\operatorname{Gr}(4,d)$ for some $d$ depending on the ramification data. The number in the column of dimension is equal to $\dim(V_{\bm\lambda})^{\mathfrak g_4}$ for the corresponding ramification data $\bm\lambda$ in each row. The numbers in the column of $c=i$ equal the lower bounds computed from Corollary \ref{bound thm} with the corresponding $c$.

For a given $c$, there may exist several choices of complex conjugate pair corresponding to different pairs of $\mathfrak g_N$-weights. If the corresponding lower bounds are the same, we just write one number. For example, in the case of $(0,2)^{\otimes 2}$, $(0,1)^{\otimes 2}$ and $c=1$ of Table \ref{table:1}, the complex conjugate pair may correspond to the weights $(0,2)^{\otimes 2}$ or $(0,1)^{\otimes 2}$. However, they give the same lower bound $1$. Hence we just write $1$ for $c=1$. If the bounds are different, we write the lower bound with the conjugate pairs corresponding to the leftmost $2c$ weights first while the one with the conjugate pairs corresponding to the rightmost $2c$ weights last, in terms of the order of the ramification data displayed on each row. Since we have at most 3 cases, the possible remaining case is clear. For instance, in the case $(0,1,0)^{\otimes 4}$, $(0,0,1)^{\otimes 4}$ and $c=2$ of Table~\ref{table:2}, the two complex conjugate pairs corresponding to $(0,1,0)^{\otimes 4}$ give the lower bound $12$ while the two complex conjugate pairs corresponding to $(0,0,1)^{\otimes 4}$ give the lower bound $24$. The remaining case, where the two conjugate pairs corresponding to $(0,1,0)^{\otimes 2}$ and $(0,0,1)^{\otimes 2}$, gives the lower bound~$2$.

\subsection[The case $N=6$]{The case $\boldsymbol{N=6}$}
In what follows, we give lower bounds for ramification data consisting of fundamental weights when $N=6$. We follow the same convention as in Section~\ref{sec:n=4}.

\begin{table}[h!] \centering
 \begin{tabular}{|c|c|c|c|c|c|}
 \hline
 ramification data & dimension & $c=1$ & $c=2$ & $c=3$ & $c=4$ \\
 \hline
 $(0,0,1)^{\otimes 4}$
 & 4 & 0 & 4 & & \\
 \hline
 $(0,1,0)^{\otimes 4}$
 & 6 & 2 & 6 & & \\
 \hline
 $(1,0,0)^{\otimes 4}$ & 3 & 1 & 3 & & \\
 \hline
 $(0,0,1)^{\otimes 2},(0,1,0)^{\otimes 2}$ & 3 & 1 & 3 & &\\
 \hline
 $(0,0,1)^{\otimes 2},(1,0,0)^{\otimes 2}$ & 2 & 0 & 2 & &\\
 \hline
 $(0,1,0)^{\otimes 2},(1,0,0)^{\otimes 2}$ & 3 & 1 & 3 & &\\
 \hline
 $(0,0,1)^{\otimes 6}$ & 30 & 2 & 2 & 10 &\\
 \hline
 $(0,1,0)^{\otimes 6}$ & 130 & 8 & 14 & 36 &\\
 \hline
 $(1,0,0)^{\otimes 6}$ & 15 & 3 & 3 & 7 &\\
 \hline
 $(0,1,0)^{\otimes 2},(0,0,1)^{\otimes 4}$ & 34 & 4,2 & 0,6 & 16 &\\
 \hline
 $(0,1,0)^{\otimes 4},(0,0,1)^{\otimes 2}$ & 55 & 3,1 & 3,7 & 19 &\\
 \hline
 $(1,0,0)^{\otimes 2},(0,0,1)^{\otimes 4}$ & 16 & 2 & 0,4 & 10 &\\
 \hline
 $(1,0,0)^{\otimes 4},(0,0,1)^{\otimes 2}$ & 10 & 0,2 & 2,0 & 6 &\\
 \hline
 $(1,0,0)^{\otimes 2},(0,1,0)^{\otimes 4}$ & 46 & 2 & 6 & 18 & \\
 \hline
 $(1,0,0)^{\otimes 4},(0,1,0)^{\otimes 2}$ & 21 & 1,3 & 5,3 & 11 & \\
 \hline
 $(1,0,0)^{\otimes 2},(0,1,0)^{\otimes 2},(0,0,1)^{\otimes 2}$ & 20 & 2 & 0,4,0 & 10 & \\
 \hline
 $(0,0,1)^{\otimes 8}$ & 330 & 20 & 6 & 0 & 50\\
 \hline
 $(0,1,0)^{\otimes 8}$ & 6111 & 69 & 59 & 113 & 311\\
 \hline
 $(1,0,0)^{\otimes 8}$ & 105 & 15 & 9 & 7 & 25\\
 \hline
 $(0,1,0)^{\otimes 4},(0,0,1)^{\otimes 4}$ & 984 & 22,28 & 12,2,24 & 0,38 & 108\\
 \hline
 $(1,0,0)^{\otimes 4},(0,0,1)^{\otimes 4}$ & 116 & 6,12 & 8,2,12 & 0,10 & 32\\
 \hline
 $(1,0,0)^{\otimes 4},(0,1,0)^{\otimes 4}$ & 510 & 6,12 & 22,4,18 & 28,18 & 74\\
 \hline
 \end{tabular}
 \caption{The case $N=6$.}\label{table:2}
\end{table}

\subsection*{Acknowledgements}
The author thanks E.~Mukhin and V.~Tarasov for useful discussions. The author also thanks the referees for their comments and suggestions that substantially improved the first version of this paper. This work was partially supported by Zhejiang Province Science Foundation, grant No.~LY14A010018.

\pdfbookmark[1]{References}{ref}
\LastPageEnding

\end{document}